\documentclass{elsarticle}
\usepackage{amsmath}
\usepackage{amssymb}
\usepackage{amsthm}
\usepackage{mathrsfs}
\newtheorem{theorem}{Theorem}[section]
\newtheorem{lemma}[theorem]{Lemma}
\newtheorem{corollary}[theorem]{Corollary}

\theoremstyle{definition}
\newtheorem{definition}[theorem]{Definition}

\begin{document}
\begin{frontmatter}
\title{The (1,2)-step competition graph of a hypertournament\footnote{Corresponding author. $E$-$mail\ \ address:$ ruijuanli@sxu.edu.cn(R. Li). Research of RL is partially supported by NNSFC under no. 11401353 and NSF of Shanxi Province under no. 2016011005. Research of XZ is partially supported by NNSFC under no. 61402317.}}
\author[1]{Ruijuan Li}
\author[1]{Xiaoting An}
\author[2]{Xinhong Zhang}
\address[1]{School of Mathematical Sciences, Shanxi University, Taiyuan, Shanxi, 030006, PR China}
\address[2]{Department of Applied Mathematics, Taiyuan University of Science and Technology, 030024, PR China}

\begin{abstract}
Competition graphs were created in connected to a biological model as a means of reflecting the competition relations among the predators in the food webs and determining the smallest dimension of ecological phase space. In 2011, Factor and Merz introduced the (1,2)-step competition graph of a digraph. Given a digraph $D=(V,A)$, the (1,2)-step competition graph of $D$, denoted $C_{1,2}(D)$, is a graph on $V(D)$ where $xy\in E(C_{1,2}(D))$ if and only if there exists a vertex $z\neq x,y$ such that either $d_{D-y}(x,z)=1$ and $d_{D-x}(y,z)\leq 2$ or $d_{D-x}(y,z)=1$ and $d_{D-y}(x,z)\leq 2$. They also characterized the (1,2)-step competition graphs of tournaments and extended some results to the $(i,j)$-step competition graphs of tournaments. In this paper, the definition of the (1,2)-step competition graph of a digraph is generalized to the one of a hypertournament and the $(1,2)$-step competition graph of a $k$-hypertournament is characterized. Also, the results are extended to the $(i,j)$-step competition graph of a $k$-hypertournament.
\end{abstract}


\end{frontmatter}

\section{Terminology and introduction}

Let $G=(V,E)$ be an undirected graph, or a graph for short. $V(G)$ and $E(G)$ are the vertex set and edge set of $G$, respectively. The complement $G^{c}$ of a graph $G$ is the graph with vertex set $V(G)$ in which two vertices are adjacent if and only if they are not adjacent in $G$. Let $G_{1}$ and $G_{2}$ be two graphs. The union of $G_{1}$ and $G_{2}$, denoted by $G_{1}\cup G_{2}$, is the graph with vertex set $V(G_{1})\cup V(G_{2})$ and edge set $E(G_{1})\cup E(G_{2})$.

Let $D=(V,A)$ be a directed graph, or a digraph for short. $V(D)$ and $A(D)$ are the vertex set and arc set of $D$, respectively. Let $i\geq 1$, $j\geq 1$. The $(i,j)$-step competition graph of $D$, denoted $C_{i,j}(D)$, is a graph on $V(D)$ where $xy\in E(C_{i,j}(D))$ if and only if there exists a vertex $z\neq x,y$ such that either $d_{D-y}(x,z)\leq i$ and $d_{D-x}(y,z)\leq j$ or $d_{D-x}(y,z)\leq i$ and $d_{D-y}(x,z)\leq j$. When $(i,j)=(1,1)$, it is also called the competition graph of $D$ and write $C_{1,2}(D)$ as $C(D)$ for short.

Competition graphs of digraphs were created in connected to a biological model, first introduced by Cohen \cite{cohen} as a means of reflecting the competition relations among the predators in the food webs and determining the smallest dimension of ecological phase space. An ecological food web is modeled by a digraph $D$. The species of the ecosystem is denoted by the vertices of $D$. There is an arc from a vertex $x$ to $y$ if $x$ preys on $y$. The competition graph of this ecological food web has the same vertex set as $D$. $xy$ is an edge of $C(D)$ if there is a vertex $z$ such that $z$ is the common prey of $x$ and $y$. Now competition graphs have been applied widely to the coding, channel assignment in communications, modeling of complex systems arising from study of energy and economic system, etc. A comprehensive introduction to competition graphs can be found in \cite{roberts,dutton,lundgren,guichard}.

In 1991, Hefner et al.\cite{hefner} defined the $(i,j)$ competition graph. In 1998, Fisher et al. \cite{fisher} showed the relation of the competition graph and the domination graph of a tournament. Recall that a tournament is an oriented complete graph. In 2008, Hedetniemi et al. \cite{hedetniemi} introduced $(1,2)$ domination. This was followed by Factor and Langley's introduction of the $(1,2)$-domination graph \cite{factor2}. Furthermore, in 2016, Factor and Langley \cite{factor3} studied the problem of  kings and heirs and gave a characterization of the $(2,2)$-domination graphs of tournaments. In recent years, many researcher investigated $m$-step competition graphs of some special digraphs and the competition numbers of some graphs etc. See \cite{helleloid,kim1,zhao,park,kim2}. Because of the similarities to the construction of \cite{factor2, cho}, in 2011, Factor and Merz \cite{factor1} gave the definition of the $(i,j)$-step competition graph of a digraph. They also characterized the $(1,2)$-step competition graph of a tournament and extended some results to the $(i,j)$-step competition graph of a tournament. They proved the following theorems related to this paper.

\begin{theorem}\cite{factor1}\label{factor11} A graph $G$ on $n\geq 5$ vertices is the $(1,2)$-step competition graph of some strong tournament if and only if $G$ is $K_{n}$, $K_{n}-E(P_{2})$, or $K_{n}-E(P_{3})$.\end{theorem}

\begin{theorem}\cite{factor1}\label{factor12} $G$, a graph on $n$ vertices, is the $(1,2)$-step competition graph of some tournament if and only if $G$ is one of the following graphs:

1. $K_{n}$, where $n\neq 2,3,4$;

2. $K_{n-1}\cup K_{1}$, where $n>1$;

3. $K_{n}-E(P_{3})$, where $n>2$;

4. $K_{n}-E(P_{2})$, where $n\neq 1,4$, or

5. $K_{n}-E(K_{3})$, where $n\geq 3$.
\end{theorem}

\begin{theorem}\cite{factor1}\label{factor13} If $T$ is a tournament with $n$ vertices, $i\geq 1$ and $j\geq 2$, then $C_{i,j}(T)=C_{1,2}(T)$.\end{theorem}

To model more general ecological food webs, the (1,2)-step competition graphs of pure local tournaments are characterized by Zhang and Li \cite{zhang} in 2016. In this paper, we study the (1,2)-step competition graph of a $k$-hypertournament and extend Theorem \ref{factor11}, \ref{factor12}, \ref{factor13} to a $k$-hypertournament.

Given two integers $n$ and $k$, $n\geq k>1$, a $k$-hypertournament on $n$ vertices is a pair $(V,A)$, where $V$ is a set of vertices, $|V|=n$ and $A$ is a set of $k$-tuples of vertices, called arcs, so that for any $k$-subset $S$ of $V$, $A$ contains exactly one of the $k!$ $k$-tuples whose entries belong to $S$. As usual, we use $V(T)$ and $A(T)$ to denote the vertex set and the arc set of $T$, respectively. Clearly, a 2-hypertournament is merely a tournament. When $k=n$, the hypertournament has only one arc and it does not have much significance to study. Thus, in what follows, we consider $3\leq k\leq n-1$.

Let $T=(V,A)$ be a $k$-hypertournament on $n$ vertices. For an arc $a$ of $T$, $T-a$ denotes a hyperdigraph obtained from $T$ by removing the arc $a$ and $\bar{a}$ denotes the set of vertices contained in $a$. If $v_{i},v_{j}\in \bar{a}$ and $v_{i}$ precedes $v_{j}$ in $a$, we say that $v_{i}$ dominates $v_{j}$ in $a$. We also say the vertex $v_{j}$ is an out-neighbour of $v_{i}$ and use the following notation: $$N_{T}^{+}(v_{i})=\{v_{j}\in V\setminus \{v_{i}\}: v_{i} \,\,\mbox{precedes}\,\, v_{j} \,\mbox{in an arc}\, a\in A(T)\}.$$
We will omit the subscript $T$ if the $k$-hypertournament is known from the context.

A path $P$ in a $k$-hypertournament $T$ is a sequence $v_{1}a_{1}v_{2}a_{2}v_{3}\cdots v_{t-1}a_{t-1}v_{t}$ of distinct vertices $v_{1},v_{2},\cdots,v_{t}$, $t\geq1$, and distinct arcs $a_{1},a_{2},\cdots,a_{t-1}$ such that $v_{i}$ precedes $v_{i+1}$ in $a_{i}$, $1\leq i\leq t-1$. Meanwhile, let the vertex set $V(P)=\{v_{1},v_{2},\cdots,v_{t}\}$ and the arc set $A(P)=\{a_{1},a_{2},\cdots,a_{t-1}\}$. The length of a path $P$ is the number of its arcs, denoted $l(P)$. A path from $x$ to $y$ is an $(x,y)$-path. The $k$-hypertournament $T$ is called strong if $T$ has an $(x,y)$-path for every pair $x,y$ of distinct vertices in $T$.

A $k$-hypertournament $T$ is said to be transitive if its vertices are labeled $v_{1},v_{2},\cdots,v_{n}$ in such an order so that $i<j$ if and only if $v_{i}$ precedes $v_{j}$ in each arc containing $v_{i}$ and $v_{j}$.

Now we generalize the $(1,2)$-step competition graph of a digraph to the one of a $k$-hypertournament.

\begin{definition}\label{1} The {\it $(i,j)$-step competition graph} of a $k$-hypertournament $T$ with $i\geq 1$ and $j\geq 1$, denoted $C_{i,j}(T)$, is a graph on $V(T)$ where $xy\in E(C_{i,j}(T))$ if and only if there exist a vertex $z\neq x,y$ and an $(x,z)$-path $P$ and a $(y,z)$-path $Q$ satisfying the following:

    (a) $y\notin V(P)$, $x\notin V(Q)$;

    (b) $l(P)\leq i$ and $l(Q)\leq j$, or $l(Q)\leq i$ and $l(P)\leq j$;

    (c) $P$ and $Q$ are arc-disjoint.
\end{definition}

If $xy\in E(C_{i,j}(T))$, we say that $x$ and $y$ $(i,j)$-step compete. In particular, we say that $x$ and $y$ compete if $l(P)=1$ and $l(Q)=1$ in $(b)$. $C_{1,1}(T)$ is also called the competition graph of the $k$-hypertournament $T$. Clearly, when $k=2$, $T$ is a tournament and $C_{i,j}(T)$ is the $(i,j)$-step competition graph of $T$.

$k$-hypertournaments form one of the most interesting class of digraphs. For the class of $k$-hypertournaments, the popular topics are the Hamiltonicity and vertex-pancyclicity. See \cite{gutin,petrovicl,yang,li}. Besides, some researchers investigated the degree sequences and score sequences of $k$-hypertournaments. See\cite{zhou,wang}. Recently, the $H$-force set of a $k$-hypertournament was also studied. See \cite{rui}. Now we consider the $(1,2)$-step competition graphs of $k$-hypertournaments.

In Section $2$ and Section $3$, useful lemmas are provided in order to make the proof of the main results easier. In Section $4$ and Section $5$, the $(1,2)$-step competition graph of a (strong) $k$-hypertournament is characterized. In Section $6$, the main results are extended to the $(i,j)$-step competition graph of a $k$-hypertournament.

\section{The missing edges of $C_{1,2}(T)$}

For a pair of distinct vertices $x$ and $y$ in $T$, $A_{T}(x,y)$ denotes the set of all arcs of $T$ in which $x$ precedes $y$, $A_{T}\{x,y\}$ denotes the set of all arcs containing $x, y$ in $T$ and $A_{T}^{*}\{x,y\}$ denotes the set of all arcs containing $x, y$ in $T$ and in which neither $x$ nor $y$ is the last entry.

\begin{lemma}\label{bothside} Let $T$ be a $k$-hypertournament with $n$ vertices, where $3\leq k\leq n-1$. Then $xy\notin E(C_{1,2}(T))$ if and only if one of the following holds:

(a) $N^{+}(x)=\emptyset$;

(b) $N^{+}(y)=\emptyset$;

(c) $N^{+}(x)=\{y\}$;

(d) $N^{+}(y)=\{x\}$;

(e) $A_{T}^{*}\{x,y\}$ contains exactly an arc $a$, and $N_{T-a}^{+}(x)\subseteq \{y\}$, $N_{T-a}^{+}(y)\subseteq \{x\}$.\end{lemma}

\begin{proof} First, we show the ``if" part. Clearly, if one of $(a)-(d)$ holds, we have $xy\notin E(C_{1,2}(T))$. Now we assume that the argument $(e)$ holds. Since $A_{T}^{*}\{x,y\}$ contains exactly an arc $a$, and $N_{T-a}^{+}(x)\subseteq \{y\}$, $N_{T-a}^{+}(y)\subseteq \{x\}$, we have to use the unique arc $a$ to obtain the out-neighbour except $y$ of $x$ and the out-neighbour except $x$ of $y$. So $x$ and $y$ are impossible to (1,2)-step compete and hence $xy\notin E(C_{1,2}(T))$.

Now we show the ``only if" part. Assume that $xy\notin E(C_{1,2}(T))$. Also, assume that $x$ and $y$ do not satisfy $(a)-(d)$. That means  $N^{+}(x)\setminus \{y\}\neq\emptyset$, $N^{+}(y)\setminus \{x\}\neq\emptyset$. Now we show that $x$ and $y$ satisfy $(e)$. Suppose $A_{T}^{*}\{x,y\}$ consists of at least two arcs, say $a_{1},a_{2}\in A_{T}^{*}\{x,y\}$. Let $w_{i}$ be the last entry of $a_{i}$ for $i=1,2$. If $w_{1}=w_{2}$, then $x$ and $y$ compete, a contradiction. So assume $w_{1}\neq w_{2}$. Note that $\binom{n-2}{k-2}-2\leq |A_{T}\{w_{1},w_{2}\}\setminus \{a_{1},a_{2}\}|\leq \binom{n-2}{k-2}$. For $(n,k)\neq(4,3)$, we have $|A_{T}\{w_{1},w_{2}\}\setminus \{a_{1},a_{2}\}|\geq 1$. For $(n,k)=(4,3)$, since both $a_{1}$ and $a_{2}$ contain $x,y$ and $a_{1}\neq a_{2}$, we have $a_{1},a_{2}\notin A_{T}\{w_{1},w_{2}\}$ and hence $|A_{T}\{w_{1},w_{2}\}\setminus \{a_{1},a_{2}\}|=|A_{T}\{w_{1},w_{2}\}|=2$. Let $a_{3}\in A_{T}\{w_{1},w_{2}\}\setminus \{a_{1},a_{2}\}$. W.l.o.g., $a_{3}\in A_{T}(w_{1},w_{2})$. Then $P=xa_{1}w_{1}a_{3}w_{2}$ and $Q=ya_{2}w_{2}$ are the paths such that $x$ and $y$ (1,2)-step compete, a contradiction. So $A_{T}^{*}\{x,y\}$ consists of exactly an arc $a$. Suppose $N_{T-a}^{+}(x)\setminus \{y\}\neq\emptyset$. Then there exists an arc $b$ distinct from $a$ such that $x$ has an out-neighbour distinct from $y$. Similarly to the proof above, whether or not the last entries of $a$ and $b$ are same, we always have $xy\in E(C_{1,2}(T))$, a contradiction. So $N_{T-a}^{+}(x)\subseteq \{y\}$. Similarly, $N_{T-a}^{+}(y)\subseteq \{x\}$. Thus, $A_{T}^{*}\{x,y\}$ contains exactly an arc $a$, and $N_{T-a}^{+}(x)\subseteq \{y\}$, $N_{T-a}^{+}(y)\subseteq \{x\}$.

The lemma holds.
\end{proof}

By the proof of Lemma \ref{bothside}, we obtain the following result.

\begin{corollary}\label{both} Let $T$ be a strong $k$-hypertournament with $n$ vertices, where $3\leq k\leq n-1$. Then $xy\notin E(C_{1,2}(T))$ if and only if one of the following holds:

(a) $N^{+}(x)=\{y\}$;

(b) $N^{+}(y)=\{x\}$;

(c) $A_{T}^{*}\{x,y\}$ contains exactly an arc $a$, and $N_{T-a}^{+}(x)\subseteq \{y\}$, $N_{T-a}^{+}(y)\subseteq \{x\}$.\end{corollary}

\section{The forbidden subgraphs of $(C_{1,2}(T))^{c}$}

\begin{lemma}\label{one} Let $G$ on $n$ vertices be the $(1,2)$-step competition graph of some $k$-hypertournament $T$, where $3\leq k\leq n-1$. Then the complement $G^{c}$ of $G$ does not contain a pair of disjoint edges.\end{lemma}

\begin{proof} Suppose the complement $G^{c}$ of $G$ contains a pair of disjoint edges, say $xy$ and $zw$. So $xy, zw\notin E(G)$ and $x, y, z, w$ are distinct. By Lemma \ref{bothside}, we have $xy$ satisfies one of the cases $(a)-(e)$ and $zw$ satisfies one of the cases $(a)-(e)$.

Suppose that at least one of $xy$ and $zw$ satisfies one of the cases $(a)-(d)$. W.l.o.g., we assume that $N^{+}(x)\subseteq\{y\}$. Then it must be true that the vertex $z$ dominates $x$ in each arc containing $x,z$ but not containing $w$. Meanwhile, it must be true that the vertex $w$ dominates $x$ in each arc containing $x,w$ but not containing $z$. So $z$ and $w$ compete and hence $zw\in E(C_{1,2}(T))=E(G)$, a contradiction. Thus, both $xy$ and $zw$ satisfy $(e)$.


However, since $A_{T}^{*}\{x,y\}$ contains exactly an arc $a$, and $N_{T-a}^{+}(x)\subseteq \{y\}$, $N_{T-a}^{+}(y)\subseteq \{x\}$, we have the vertex $x$ must be the last entry in each arc containing $x,z,w$ but not containing $y$ and the vertex $y$ must be the last entry in each arc containing $y,z,w$ but not containing $x$. Thus, $A_{T}^{*}\{z,w\}$ contains at least two arcs, which contradicts the fact that $zw$ satisfies $(e)$.

The lemma holds.\end{proof}

\begin{lemma}\label{two} Let $G$ on $n$ vertices be the $(1,2)$-step competition graph of some $k$-hypertournament $T$, where $3\leq k\leq n-1$. Then the complement $G^{c}$ of $G$ does not contain 3-cycle.\end{lemma}

\begin{proof} Suppose to the contrary that the complement $G^{c}$ of $G$ contains 3-cycle, say $xyzx$. So $xy, xz, yz\notin E(G)$ and $x, y, z$ are distinct. By Lemma \ref{bothside}, we have $xy$, $xz$ and $yz$ satisfy one of the cases $(a)-(e)$, respectively.

\vskip 0.1cm
{\noindent\bf Claim 1.} None of $xy$, $xz$ and $yz$ satisfies the case $(a)$ or $(b)$.
\vskip 0.1cm

\begin{proof} Suppose at least one of $xy$, $xz$ and $yz$ satisfies the case $(a)$ or $(b)$. W.l.o.g., we assume that $xy$ satisfies $(a)$, i.e., $N^{+}(x)=\emptyset$. Then it must be true that the vertex $y$ dominates $x$ in each arc containing $x,y$ but not containing $z$. Also, it must be true that the vertex $z$ dominates $x$ in each arc containing $x,z$ but not containing $y$. So $y$ and $z$ compete and $yz\in E(C_{1,2}(T))=E(G)$, a contradiction. Thus, none of $xy$, $xz$ and $yz$ satisfies the case $(a)$ or $(b)$.\end{proof}

\vskip 0.1cm
{\noindent\bf Claim 2.} At most one of $xy$, $xz$ and $yz$ satisfies the case $(c)$ or $(d)$.
\vskip 0.1cm

\begin{proof} Suppose at least two edges among $xy$, $xz$ and $yz$ satisfy the case $(c)$ or $(d)$. W.l.o.g., assume that both $xy$ and $yz$ satisfy $(c)$ or $(d)$. We consider the following four cases.

{\it Case 1:} Both $xy$ and $yz$ satisfy $(c)$. It means that $N^{+}(x)=\{y\}$, $N^{+}(y)=\{z\}$. If $xz$ satisfies $(c)$, i.e., $N^{+}(x)=\{z\}$, contradicting $N^{+}(x)=\{y\}$. If $xz$ satisfies $(d)$, i.e., $N^{+}(z)=\{x\}$. Now the arcs containing simultaneously $x, y, z$ do not satisfy $N^{+}(x)=\{y\}$, $N^{+}(y)=\{z\}$ and $N^{+}(z)=\{x\}$, a contradiction. If $xz$ satisfies $(e)$, i.e., $A_{T}^{*}\{x,z\}$ contains exactly an arc $a$, then there exists a vertex $w$ such that $w\in N^{+}(x)$. Since $N^{+}(x)=\{y\}$, we have $w=y$. So $x$ is the second last entry, $y$ is the last entry and $z$ is in any other entry in $a$. Then the vertex $z$ dominates $y$ in $a$. Also, the vertex $x$ dominates $y$ in each arc containing $x,y$ but not containing $z$. So $xz\in E(C_{1,2}(T))=E(G)$, a contradiction.

{\it Case 2:} Both $xy$ and $yz$ satisfy $(d)$. It means that $N^{+}(y)=\{x\}$, $N^{+}(z)=\{y\}$. Similarly to Case 1, we can also get a contradiction.

{\it Case 3:}  $xy$ satisfies $(c)$ and $yz$ satisfies $(d)$. It means that $N^{+}(x)=\{y\}$, $N^{+}(z)=\{y\}$. Now the arcs containing simultaneously $x, y, z$ do not satisfy $N^{+}(x)=\{y\}$, $N^{+}(z)=\{y\}$, a contradiction.

{\it Case 4:}  $xy$ satisfies $(d)$ and $yz$ satisfies $(c)$. It means that $N^{+}(y)=\{x\}$ and $N^{+}(y)=\{z\}$. Then $x=z$, a contradiction.

Thus, at most one of $xy$, $xz$ and $yz$ satisfies the case $(c)$ or $(d)$.\end{proof}

\vskip 0.1cm
{\noindent\bf Claim 3.} At most one of $xy$, $xz$ and $yz$ satisfies the case $(e)$.
\vskip 0.1cm

\begin{proof} Suppose at least two edges among $xy$, $xz$ and $yz$ satisfy the case $(e)$. W.l.o.g., assume that both $xz$ and $yz$ satisfy $(e)$. From the assumption that $xz$ satisfies $(e)$, we get the vertex $y$ dominates $x$ in each arc containing $x,y$ but not containing $z$. From the assumption that $yz$ satisfies $(e)$, we get the vertex $x$ dominates $y$ in each arc containing $x,y$ but not containing $z$. This is a contradiction. Thus, at most one of $xy$, $xz$ and $yz$ satisfies the case $(e)$.\end{proof}

By Claim 1-3, it is impossible that $xy,xz,yz\notin E(G)$ hold simultaneously. Thus the complement $G^{c}$ of $G$ does not contain 3-cycle. The lemma holds.\end{proof}

\begin{lemma}\label{three} Let $G$ on $n$ vertices be the $(1,2)$-step competition graph of some $k$-hypertournament $T$, where $3\leq k\leq n-1$. Then the complement $G^{c}$ of $G$ does not contain $K_{1,3}$, unless $G=K_{n-1}\cup K_{1}$.\end{lemma}

\begin{proof} Let $T$ be a $k$-hypertournament on $n$ vertices, where $3\leq k\leq n-1$, and $G$ be the $(1,2)$-step competition graph of $T$. Assume $G\neq K_{n-1}\cup K_{1}$. Now we show that the complement $G^{c}$ of $G$ does not contain $K_{1,3}$. Suppose not. Let $\{x,y,z,w\}$ and $\{xy,xz,xw\}$ be the vertex set and edge set of the subgraph $K_{1,3}$, respectively. So $xy,xz,xw\notin E(G)$. By Lemma \ref{bothside}, we have $xy$, $xz$ and $xw$ satisfy one of the cases $(a)-(e)$.

\vskip 0.1cm
{\noindent\bf Claim 1.} None of $xy$, $xz$ and $xw$ satisfies the case $(a)$ or $(b)$.
\vskip 0.1cm

\begin{proof} {Suppose at least one of $xy$, $xz$ and $xw$ satisfies the case $(a)$ or $(b)$. W.l.o.g., assume that $xy$ satisfies $(a)$, i.e., $N^{+}(x)=\emptyset$. Let $V(T)=\{v_{1},v_{2},\cdots,v_{n}\}$ and let $x=v_{n}$. By Lemma \ref{bothside}, for all $1\leq i\leq n-1$, we have $v_{i}v_{n}\notin E(C_{1,2}(T))$. We claim that for all $1\leq i<j\leq n-1$, $v_{i}v_{j}\in E(C_{1,2}(T))$.

\hangafter 1
\hangindent 0.8em
\noindent
$\bullet$ For $1\leq i<j\leq n-(k-1)$, the vertex $v_{i}$ dominates $v_{n}$ by the arc consisting of $v_{i},\cdots,v_{i+(k-3)},v_{n-1},v_{n}$ and the vertex $v_{j}$ dominates $v_{n}$ by the arc consisting of $v_{j},\cdots,v_{j+(k-3)},v_{n-1},v_{n}$. Then $v_{i}$ and $v_{j}$ compete and $v_{i}v_{j}\in E(C_{1,2}(T))$.

\hangafter 1
\hangindent 0.8em
\noindent
$\bullet$ For $n-(k-2)\leq i<j\leq n-1$, the vertex $v_{i}$ dominates $v_{n}$ by the arc consisting of $v_{n-k},\cdots,v_{n-3},v_{n-1},v_{n}$ for $n-(k-2)\leq i\leq n-3$ and by the arc consisting of $v_{1},\cdots,v_{1+(k-3)},v_{n-2},v_{n}$ for $i=n-2$ and the vertex $v_{j}$ dominates $v_{n}$ by the arc consisting of $v_{n-(k-1)},\cdots,v_{n-2},v_{n-1},v_{n}$. Then $v_{i}$ and $v_{j}$ compete and $v_{i}v_{j}\in E(C_{1,2}(T))$.

\hangafter 1
\hangindent 0.8em
\noindent
$\bullet$ For $1\leq i\leq n-k$ and $n-(k-2)\leq j\leq n-1$, the vertex $v_{i}$ dominates $v_{n}$ by the arc consisting of $v_{i},\cdots,v_{i+(k-3)},v_{n-1},v_{n}$ and the vertex $v_{j}$ dominates $v_{n}$ by the arc consisting of $v_{n-(k-1)},\cdots,v_{n-2},v_{n-1},v_{n}$. Then $v_{i}$ and $v_{j}$ compete and $v_{i}v_{j}\in E(C_{1,2}(T))$.

\hangafter 1
\hangindent 0.8em
\noindent
$\bullet$ For $i=n-(k-1)$ and $n-(k-2)\leq j\leq n-1$, the vertex $v_{n-(k-1)}$ dominates $v_{n}$ by the arc consisting of $v_{1},\cdots,v_{1+(k-3)},v_{n-2},v_{n}$ for $k=3$ and by the arc consisting of $v_{n-k},v_{n-(k-1)},\cdots,v_{n-3},v_{n-1},v_{n}$ for $4\leq k\leq n-1$ and the vertex $v_{j}$ dominates $v_{n}$ by the arc consisting of $v_{n-(k-1)},\cdots,v_{n-2},v_{n-1},v_{n}$. Then $v_{j}$ and $v_{n-(k-1)}$ compete and $v_{j}v_{n-(k-1)}\in E(C_{1,2}(T))$.}

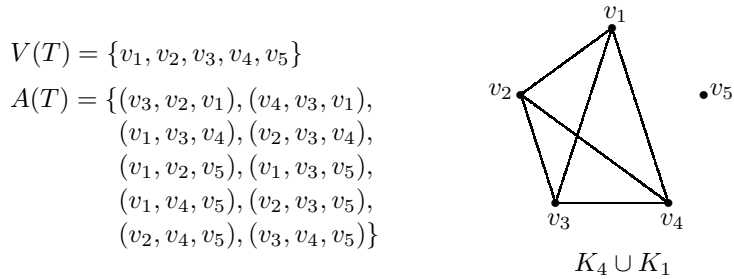
\begin{figure}[h]
\unitlength0.3cm
\begin{center}
\begin{picture}(35,12)
\put(24.17,3){\circle*{.3}}
\put(23.8,2){$v_3$}
\put(29.17,3){\circle*{.3}}
\put(28.8,2){$v_4$}
\put(22.63,7.8){\circle*{.3}}
\put(21.2,7.8){$v_2$}
\put(30.74,7.8){\circle*{.3}}
\put(31,7.8){$v_5$}
\put(26.67,10.75){\circle*{.3}}
\put(26.3,11.2){$v_1$}

\qbezier(26.67,10.75)(24.65,9.275)(22.63,7.8)
\qbezier(24.17,3)(23.4,5.4)(22.63,7.8)
\qbezier(24.17,3)(27.67,3)(29.17,3)
\qbezier(22.63,7.8)(25.9,5.4)(29.17,3)
\qbezier(26.67,10.75)(25.42,6.875)(24.17,3)
\qbezier(26.67,10.75)(27.92,6.875)(29.17,3)
\put(25,0){$K_4\cup K_1$}

\put(0,9.28){$V(T)=\{v_1,v_2,v_3,v_4,v_5\}$}
\put(0,7.28){$A(T)=\{(v_3,v_2,v_1),(v_4,v_3,v_1),$}
\put(4.8,5.78){$(v_1,v_3,v_4),(v_2,v_3,v_4),$}
\put(4.8,4.28){$(v_1,v_2,v_5),(v_1,v_3,v_5),$}
\put(4.8,2.78){$(v_1,v_4,v_5),(v_2,v_3,v_5),$}
\put(4.8,1.28){$(v_2,v_4,v_5),(v_3,v_4,v_5)\}$}
\end{picture}
\caption{A 3-hypertournament $T$ on 5 vertices and its competition graph $K_4\cup K_1$. }
\end{center}
\label{fig:t}
\end{figure}

Then $C_{1,2}(T)=K_{n-1}\cup K_{1}$, a contradiction. See an example in Figure 1. Thus, none of $xy$, $xz$ and $xw$ satisfies the case $(a)$ or $(b)$.
\end{proof}

\vskip 0.1cm
{\noindent\bf Claim 2.} At most one of $xy$, $xz$ and $xw$ satisfies the case $(c)$ or $(d)$.
\vskip 0.1cm

\begin{proof} Suppose at least two edges among $xy$, $xz$ and $xw$ satisfy the case $(c)$ or $(d)$. W.l.o.g., assume that both $xy$ and $xz$ satisfy the case $(c)$ or $(d)$. We consider the following four cases.

{\it Case 1:} Both $xy$ and $xz$ satisfy $(c)$. It means that $N^{+}(x)=\{y\}$ and $N^{+}(x)=\{z\}$. Then $y=z$, a contradiction.

{\it Case 2:} Both $xy$ and $xz$ satisfy $(d)$. It means that $N^{+}(y)=\{x\}$ and $N^{+}(z)=\{x\}$.
Now the arcs containing simultaneously $x, y, z$ do not satisfy $N^{+}(y)=\{x\}$, $N^{+}(z)=\{x\}$, a contradiction.

{\it Case 3:} $xy$ satisfies $(c)$ and $xz$ satisfies $(d)$. It means that $N^{+}(x)=\{y\}$ and $N^{+}(z)=\{x\}$. If $xw$ satisfies $(c)$, then $N^{+}(x)=\{w\}$, contradicting $N^{+}(x)=\{y\}$. If $xw$ satisfies $(d)$, then $N^{+}(w)=\{x\}$. Now the arcs containing simultaneously $x, z, w$ do not satisfy $N^{+}(z)=\{x\}$, $N^{+}(w)=\{x\}$, a contradiction. If $xw$ satisfies $(e)$. Let $b$ be an arc containing $x, z, w$. $N^{+}(x)=\{y\}$ yields $z$ dominates $x$ in $b$. $N^{+}(z)=\{x\}$ yields $z$ is the second last entry, $x$ is the last entry and $w$ is in other entry of $b$. Clearly, $b\notin A_{T}^{*}\{x,w\}$, i.e., $z\in N_{T-a}^{+}(w)$, which contradicts the fact that $N_{T-a}^{+}(w)\subseteq \{x\}$.

{\it Case 4:} $xy$ satisfies $(d)$ and $xz$ satisfies $(c)$. Similarly to Case 3, we can also get a contradiction.

Thus, at most one of $xy$, $xz$ and $xw$ satisfies the case $(c)$ or $(d)$.\end{proof}

\vskip 0.1cm
{\noindent\bf Claim 3.} At most one of $xy$, $xz$ and $xw$ satisfies the case $(e)$.
\vskip 0.1cm

\begin{proof} Suppose at least two edges among $xy$, $xz$ and $xw$ satisfy the case $(e)$. Assume that $xy$ and $xz$ satisfy $(e)$. From the assumption that $xy$ satisfies $(e)$, we get the vertex $z$ dominates $y$ in each arc containing $y,z$ but not containing $x$. From the assumption that $xz$ satisfies $(e)$, we get the vertex $y$ dominates $z$ in each arc containing $y,z$ but not containing $x$. This is a contradiction. Thus, at most one of $xy$, $xz$ and $xw$ satisfies the case $(e)$.\end{proof}

By Claim 1-3, it is impossible that $xy,xz,xw\notin E(G)$ hold simultaneously. Thus the complement $G^{c}$ of $G$ does not contain $K_{1,3}$ unless $G=K_{n-1}\cup K_{1}$. The lemma holds.\end{proof}

By Corollary \ref{both} and the proof of Lemma \ref{three}, we obtain the following result.

\begin{corollary}\label{four} Let $G$ on $n$ vertices be the $(1,2)$-step competition graph of some strong $k$-hypertournament $T$, where $3\leq k\leq n-1$. Then the complement $G^{c}$ of $G$ does not contain $K_{1,3}$.\end{corollary}

\section{Strong $k$-hypertournaments}

\begin{theorem}\label{strong} A graph $G$ on $n$ vertices is the $(1,2)$-step competition graph of some strong $k$-hypertournament $T$ with $3\leq k\leq n-1$ if and only if $G$ is $K_{n}$, $K_{n}-E(P_{2})$, or $K_{n}-E(P_{3})$.\end{theorem}

\begin{proof} {We first show the ``if" part. Let $T$ be a transitive $k$-hypertournament with the vertices $v_{1},v_{2},\cdots,v_{n}$. Let $T_{1}$ be a $k$-hypertournament obtained from $T$ by replacing the arc $(v_{1},v_{2},v_{n-(k-3)},\cdots,v_{n})$ with $(v_{n},\cdots,v_{n-(k-3)},v_{2},v_{1})$. See an example in Figure 2. It is easy to check that $T_{1}$ is strong. Now we show that $C_{1,2}(T_{1})=K_{n}-E(P_{2})$. For convenience, let $a=(v_{n},\cdots,v_{n-(k-3)},v_{2},v_{1})$. We claim that $v_{n-1}v_{n}\notin E(C_{1,2}(T_{1}))$. Indeed, for $k=3$, the vertex $v_{n-1}$ has a unique out-neighbour $v_{n}$ and Corollary \ref{both} $(a)$ implies $v_{n-1}v_{n}\notin E(C_{1,2}(T_{1}))$. For $4\leq k\leq n-1$, $A_{T_{1}}^{*}\{v_{n-1},v_{n}\}$ contains exactly an arc $a$, and $N_{T_{1}-a}^{+}(v_{n-1})=\{v_{n}\}$, $N_{T_{1}-a}^{+}(v_{n})=\emptyset$. Corollary \ref{both} $(c)$ implies $v_{n-1}v_{n}\notin E(C_{1,2}(T_{1}))$. We also claim $v_{i}v_{j}\in E(C_{1,2}(T_{1}))$ for all $\{i,j\}\neq \{n-1,n\}$. W.l.o.g., assume $i<j$.

\begin{figure}[h]
\unitlength0.3cm
\begin{center}
\begin{picture}(35,12)
\put(24.17,3){\circle*{.3}}
\put(23.8,2){$v_3$}
\put(29.17,3){\circle*{.3}}
\put(28.8,2){$v_4$}
\put(22.63,7.8){\circle*{.3}}
\put(21.2,7.8){$v_2$}
\put(30.74,7.8){\circle*{.3}}
\put(31,7.8){$v_5$}
\put(26.67,10.75){\circle*{.3}}
\put(26.3,11.2){$v_1$}

\qbezier(26.67,10.75)(24.65,9.275)(22.63,7.8)
\qbezier(24.17,3)(23.4,5.4)(22.63,7.8)
\qbezier(24.17,3)(27.67,3)(29.17,3)
\qbezier(26.67,10.75)(28.705,9.275)(30.74,7.8)
\qbezier(22.63,7.8)(25.9,5.4)(29.17,3)
\qbezier(22.63,7.8)(26.69,7.8)(30.75,7.8)
\qbezier(24.17,3)(27.455,5.4)(30.74,7.8)
\qbezier(26.67,10.75)(25.42,6.875)(24.17,3)
\qbezier(26.67,10.75)(27.92,6.875)(29.17,3)
\put(25,0){$K_5-P_2$}

\put(0,9.28){$V(T_1)=\{v_1,v_2,v_3,v_4,v_5\}$}
\put(0,7.28){$A(T_1)=\{(v_1,v_2,v_3),(v_1,v_2,v_4),$}
\put(4.8,5.78){$(v_1,v_3,v_4),(v_1,v_3,v_5),$}
\put(4.8,4.28){$(v_1,v_4,v_5),(v_2,v_3,v_4),$}
\put(4.8,2.78){$(v_2,v_3,v_5),(v_2,v_4,v_5),$}
\put(4.8,1.28){$(v_3,v_4,v_5),(v_5,v_1,v_2)\}$}
\end{picture}
\caption{A strong 3-hypertournament $T_1$ on 5 vertices and its competition graph $K_5-P_2$. }
\end{center}
\label{Tone}
\end{figure}
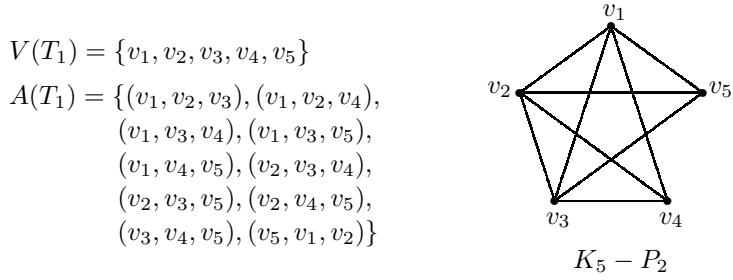

\hangafter 1
\hangindent 0.8em
\noindent
$\bullet$ For $1\leq i<j\leq n-(k-1)$, $v_{i}$ dominates $v_{n}$ by the arc $(v_{i},v_{n-(k-2)},\cdots,v_{n-1},v_{n})$ and $v_{j}$ dominates $v_{n}$ by the arc $(v_{j},v_{n-(k-2)},\cdots,v_{n-1},v_{n})$. Then $v_{i}$ and $v_{j}$ compete and $v_{i}v_{j}\in E(C_{1,2}(T_{1}))$.

\hangafter 1
\hangindent 0.8em
\noindent
$\bullet$ For $n-(k-2)\leq i<j\leq n-1$, $v_{i}$ dominates $v_{n}$ by the arc $(v_{1},v_{n-(k-2)},\cdots,v_{i},\\\cdots,v_{j},\cdots,v_{n-1},v_{n})$ and $v_{j}$ dominates $v_{n}$ by the arc $(v_{2},v_{n-(k-2)},\cdots,v_{i},\cdots,\\v_{j},\cdots,v_{n-1},v_{n})$. Then $v_{i}$ and $v_{j}$ compete and $v_{i}v_{j}\in E(C_{1,2}(T_{1}))$.

\hangafter 1
\hangindent 0.8em
\noindent
$\bullet$ For $i=1$ and $n-(k-2)\leq j\leq n-1$, $v_{1}$ dominates $v_{n}$ by the arc $(v_{1},v_{n-(k-2)},\cdots,v_{n-1},v_{n})$ and $v_{j}$ dominates $v_{n}$ by the arc $(v_{2},v_{n-(k-2)},\cdots,\\v_{j},\cdots,v_{n-1},v_{n})$. Then $v_{1}$ and $v_{j}$ compete and $v_{1}v_{j}\in E(C_{1,2}(T_{1}))$.

\hangafter 1
\hangindent 0.8em
\noindent
$\bullet$ For $2\leq i\leq n-(k-1)$ and $n-(k-2)\leq j\leq n-1$, $v_{i}$ dominates $v_{n}$ by the arc $(v_{i},v_{n-(k-2)},\cdots,v_{n-1},v_{n})$ and $v_{j}$ dominates $v_{n}$ by the arc $(v_{1},v_{n-(k-2)},\cdots,v_{j},\cdots,v_{n-1},v_{n})$. Then $v_{i}$ and $v_{j}$ compete and $v_{i}v_{j}\in E(C_{1,2}(T_{1}))$.

\hangafter 1
\hangindent 0.8em
\noindent
$\bullet$ For $i=1$ and $j=n$, $v_{n}$ dominates $v_{2}$ by the arc $a$, $v_{2}$ dominates $v_{n-1}$ by the arc $(v_{2},v_{n-(k-2)},\cdots,v_{n-1},v_{n})$ and $v_{1}$ dominates $v_{n-1}$ by the arc $(v_{1},v_{n-(k-2)},\cdots,v_{n-1},v_{n})$. Then $v_{1}$ and $v_{n}$ (1,2)-step compete and $v_{1}v_{n}\in E(C_{1,2}(T_{1}))$.

\hangafter 1
\hangindent 0.8em
\noindent
$\bullet$ For $2\leq i\leq n-2$ and $j=n$, $v_{n}$ dominates $v_{1}$ by the arc $a$, $v_{1}$ dominates $v_{n-1}$ by the arc $(v_{1},v_{n-(k-2)},\cdots,v_{n-1},v_{n})$ and $v_{i}$ dominates $v_{n-1}$ by $(v_{i},v_{n-(k-2)},\cdots,v_{n-1},v_{n})$ for $2\leq i\leq n-(k-1)$ and by $(v_{2},v_{n-(k-2)},\cdots,v_{i},\\\cdots,v_{n-1},v_{n})$ for $n-(k-2)\leq i\leq n-2$. Then $v_{i}$ and $v_{n}$ (1,2)-step compete and $v_{i}v_{n}\in E(C_{1,2}(T_{1}))$.}

Thus $C_{1,2}(T_{1})=K_{n}-E(P_{2})$.

{Let $T_{2}$ be a $k$-hypertournament obtained from $T_{1}$ above by replacing the arc $(v_{1},v_{2},v_{n-(k-2)},\cdots,v_{n-1})$ with $(v_{n-1},\cdots,v_{n-(k-2)},v_{2},v_{1})$. See an example in Figure 3. It is easy to check that $T_{2}$ is strong. Now we show that $C_{1,2}(T_{2})=K_{n}$.

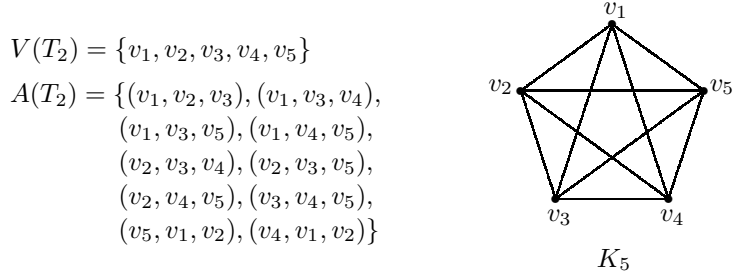
\begin{figure}[h]
\unitlength0.3cm
\begin{center}
\begin{picture}(35,12)
\put(24.17,3){\circle*{.3}}
\put(23.8,2){$v_3$}
\put(29.17,3){\circle*{.3}}
\put(28.8,2){$v_4$}
\put(22.63,7.8){\circle*{.3}}
\put(21.2,7.8){$v_2$}
\put(30.74,7.8){\circle*{.3}}
\put(31,7.8){$v_5$}
\put(26.67,10.75){\circle*{.3}}
\put(26.3,11.2){$v_1$}

\qbezier(26.67,10.75)(24.65,9.275)(22.63,7.8)
\qbezier(24.17,3)(23.4,5.4)(22.63,7.8)
\qbezier(24.17,3)(27.67,3)(29.17,3)
\qbezier(29.17,3)(29.955,5.4)(30.74,7.8)
\qbezier(26.67,10.75)(28.705,9.275)(30.74,7.8)
\qbezier(22.63,7.8)(25.9,5.4)(29.17,3)
\qbezier(22.63,7.8)(26.69,7.8)(30.75,7.8)
\qbezier(24.17,3)(27.455,5.4)(30.74,7.8)
\qbezier(26.67,10.75)(25.42,6.875)(24.17,3)
\qbezier(26.67,10.75)(27.92,6.875)(29.17,3)
\put(26,0){$K_5$}

\put(0,9.28){$V(T_2)=\{v_1,v_2,v_3,v_4,v_5\}$}
\put(0,7.28){$A(T_2)=\{(v_1,v_2,v_3),(v_1,v_3,v_4),$}
\put(4.8,5.78){$(v_1,v_3,v_5),(v_1,v_4,v_5),$}
\put(4.8,4.28){$(v_2,v_3,v_4),(v_2,v_3,v_5),$}
\put(4.8,2.78){$(v_2,v_4,v_5),(v_3,v_4,v_5),$}
\put(4.8,1.28){$(v_5,v_1,v_2),(v_4,v_1,v_2)\}$}
\end{picture}
\caption{A strong 3-hypertournament $T_2$ on 5 vertices and its competition graph $K_5$. }
\end{center}
\label{Ttwo}
\end{figure}

\hangafter 1
\hangindent 0.8em
\noindent
$\bullet$ For $\{i,j\}\neq \{n-1,n\}$, similarly to the proof of $T_{1}$, we have $v_{i}v_{j}\in E(C_{1,2}(T_{2}))$.

\hangafter 1
\hangindent 0.8em
\noindent
$\bullet$ For $i=n-1$ and $j=n$, $v_{n-1}$ dominates $v_{1}$ by the arc $(v_{n-1},\cdots,v_{n-(k-2)},v_{2},v_{1})$ and $v_{n}$ dominates $v_{1}$ by the arc $(v_{n},\cdots,v_{n-(k-3)},v_{2},v_{1})$. Then $v_{n-1}$ and $v_{n}$ compete and $v_{n-1}v_{n}\in E(C_{1,2}(T_{2}))$.}

Thus $C_{1,2}(T_{2})=K_{n}$.

{Let $T_{3}$ be a $k$-hypertournament with the vertices $v_{1},v_{2},\cdots,v_{n}$ satisfying the following:

1. each arc excluding $v_{1},v_{2}$ satisfies $i<j$ if and only if $v_{i}$ precedes $v_{j}$;

2. each arc including $v_{1},v_{2}$ satisfies that $v_{1}$ is the second last entry, $v_{2}$ is the last entry and the remaining $k-2$ entries satisfy $i<j$ if and only if $v_{i}$ precedes $v_{j}$;

3. each arc including $v_{1}$ but excluding $v_{2}$ satisfies that $v_{1}$ is the last entry and the remaining $k-1$ entries satisfy $i<j$ if and only if $v_{i}$ precedes $v_{j}$;

4. each arc including $v_{2}$ but excluding $v_{1},v_{3}$ satisfies that $v_{2}$ is the last entry and the remaining $k-1$ entries satisfy $i<j$ if and only if $v_{i}$ precedes $v_{j}$;

5. each arc including $v_{2},v_{3}$ but excluding $v_{1}$ satisfies that $v_{2}$ is the second last entry, $v_{3}$ is the last entry and the remaining $k-2$ entries satisfy $i<j$ if and only if $v_{i}$ precedes $v_{j}$.

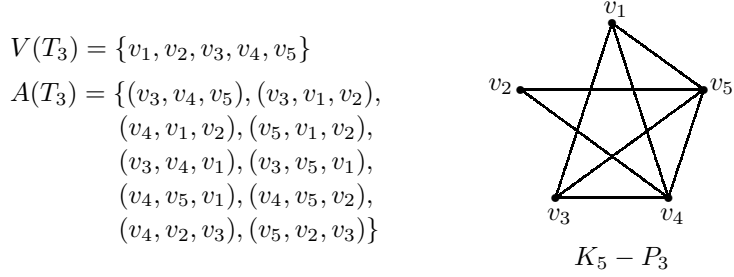
\begin{figure}[h]
\unitlength0.3cm
\begin{center}
\begin{picture}(35,12)
\put(24.17,3){\circle*{.3}}
\put(23.8,2){$v_3$}
\put(29.17,3){\circle*{.3}}
\put(28.8,2){$v_4$}
\put(22.63,7.8){\circle*{.3}}
\put(21.2,7.8){$v_2$}
\put(30.74,7.8){\circle*{.3}}
\put(31,7.8){$v_5$}
\put(26.67,10.75){\circle*{.3}}
\put(26.3,11.2){$v_1$}

\qbezier(24.17,3)(27.67,3)(29.17,3)
\qbezier(29.17,3)(29.955,5.4)(30.74,7.8)
\qbezier(26.67,10.75)(28.705,9.275)(30.74,7.8)
\qbezier(22.63,7.8)(25.9,5.4)(29.17,3)
\qbezier(22.63,7.8)(26.69,7.8)(30.75,7.8)
\qbezier(24.17,3)(27.455,5.4)(30.74,7.8)
\qbezier(26.67,10.75)(25.42,6.875)(24.17,3)
\qbezier(26.67,10.75)(27.92,6.875)(29.17,3)
\put(25,0){$K_5-P_3$}

\put(0,9.28){$V(T_3)=\{v_1,v_2,v_3,v_4,v_5\}$}
\put(0,7.28){$A(T_3)=\{(v_3,v_4,v_5),(v_3,v_1,v_2),$}
\put(4.8,5.78){$(v_4,v_1,v_2),(v_5,v_1,v_2),$}
\put(4.8,4.28){$(v_3,v_4,v_1),(v_3,v_5,v_1),$}
\put(4.8,2.78){$(v_4,v_5,v_1),(v_4,v_5,v_2),$}
\put(4.8,1.28){$(v_4,v_2,v_3),(v_5,v_2,v_3)\}$}
\end{picture}
\caption{A strong 3-hypertournament $T_3$ on 5 vertices and its competition graph $K_5-P_3$. }
\end{center}
\label{Tthree}
\end{figure}

See an example in Figure 4. It is easy to check that $T_{3}$ is strong. Now we show that $C_{1,2}(T_{3})=K_{n}-E(P_{3})$. Note that $N^{+}(v_{1})=\{v_{2}\}$ and $N^{+}(v_{2})=\{v_{3}\}$. By Corollary \ref{both} $(a)$, we have $v_{1}v_{2}, v_{2}v_{3}\notin E(C_{1,2}(T_{3}))$. Now we consider the arc $v_{i}v_{j}$ for $\{i,j\}\neq\{1,2\}$ and $\{i,j\}\neq\{2,3\}$. W.l.o.g., assume $i<j$.

\hangafter 1
\hangindent 0.8em
\noindent
$\bullet$ For $3\leq i<j\leq n-(k-3)$, $v_{i}$ dominates $v_{2}$ by the arc $(v_{i},\cdots,v_{i+(k-3)},v_{1},v_{2})$ and $v_{j}$ dominates $v_{2}$ by the arc $(v_{j},\cdots,v_{j+(k-3)},v_{1},v_{2})$. Then $v_{i}$ and $v_{j}$ compete and $v_{i}v_{j}\in E(C_{1,2}(T_{3}))$.

\hangafter 1
\hangindent 0.8em
\noindent
$\bullet$ For $n-(k-4)\leq i<j\leq n$, $v_{i}$ dominates $v_{2}$ by the arc $(v_{n-(k-2)},\cdots,v_{i},\cdots,v_{j},\\\cdots,v_{n-1},v_{1},v_{2})$ and $v_{j}$ dominates $v_{2}$ by the arc $(v_{n-(k-3)},\cdots,v_{i},\cdots,v_{j},\cdots,\\v_{n},v_{1},v_{2})$. Then $v_{i}$ and $v_{j}$ compete and $v_{i}v_{j}\in E(C_{1,2}(T_{3}))$.

\hangafter 1
\hangindent 0.8em
\noindent
$\bullet$ For $3\leq i\leq n-(k-3)$ and $n-(k-4)\leq j\leq n$, $v_{i}$ dominates $v_{2}$ by the arc $(v_{i},\cdots,v_{i+(k-3)},v_{1},v_{2})$ for $3\leq i\leq n-(k-2)$ and by the arc $(v_{n-(k-2)},v_{n-(k-3)},\cdots,v_{n-1},v_{1},v_{2})$ for $i=n-(k-3)$ and $v_{j}$ dominates $v_{2}$ by the arc $(v_{n-(k-3)},\cdots,v_{j},\cdots,v_{n},v_{1},v_{2})$. Then $v_{i}$ and $v_{j}$ compete and $v_{i}v_{j}\in E(C_{1,2}(T_{3}))$.

\hangafter 1
\hangindent 0.8em
\noindent
$\bullet$ For $i=1$ and $3\leq j\leq n-(k-2)$, $v_{1}$ dominates $v_{2}$ by the arc $(v_{n-(k-3)},\cdots,v_{n},\\v_{1},v_{2})$ and $v_{j}$ dominates $v_{2}$ by the arc $(v_{j},\cdots,v_{j+(k-3)},v_{1},v_{2})$. Then $v_{1}$ and $v_{j}$ compete and $v_{1}v_{j}\in E(C_{1,2}(T_{3}))$.

\hangafter 1
\hangindent 0.8em
\noindent
$\bullet$ For $i=1$ and $n-(k-3)\leq j\leq n$, $v_{1}$ dominates $v_{2}$ by the arc $(v_{3},\cdots,v_{k},v_{1},v_{2})$ and $v_{j}$ dominates $v_{2}$ by the arc $(v_{n-(k-3)},\cdots,v_{j},\cdots,v_{n},v_{1},v_{2})$. Then $v_{1}$ and $v_{j}$ compete and $v_{1}v_{j}\in E(C_{1,2}(T_{3}))$.

\hangafter 1
\hangindent 0.8em
\noindent
$\bullet$ For $i=2$ and $4\leq j\leq n-(k-3)$, $v_{2}$ dominates $v_{3}$ by the arc $(v_{n-(k-3)},\cdots,v_{n},\\v_{2},v_{3})$, $v_{3}$ dominates $v_{1}$ by the arc $(v_{3},\cdots,v_{k},v_{1},v_{2})$ and $v_{j}$ dominates $v_{1}$ by the arc $(v_{j},\cdots,v_{j+(k-3)},v_{1},v_{2})$. Then $v_{2}$ and $v_{j}$ (1,2)-step compete and $v_{2}v_{j}\in E(C_{1,2}(T_{3}))$.

\hangafter 1
\hangindent 0.8em
\noindent
$\bullet$ For $i=2$ and $n-(k-4)\leq j\leq n$, $v_{2}$ dominates $v_{3}$ by the arc $(v_{n-(k-3)},\cdots,v_{n},\\v_{2},v_{3})$, $v_{3}$ dominates $v_{1}$ by the arc $(v_{3},\cdots,v_{k},v_{1},v_{2})$ and $v_{j}$ dominates $v_{1}$ by the arc $(v_{n-(k-3)},\cdots,v_{j},\cdots,v_{n},v_{1},v_{2})$. Then $v_{2}$ and $v_{j}$ (1,2)-step compete and $v_{2}v_{j}\in E(C_{1,2}(T_{3}))$.

}

Thus $C_{1,2}(T_{3})=K_{n}-E(P_{3})$.

Now we show the ``only if" part. Let $T$ be a strong $k$-hypertournament and $G$ be the $(1,2)$-step competition graph of $T$. We show that $G$ is $K_{n}$, $K_{n}-E(P_{2})$, or $K_{n}-E(P_{3})$. We claim that $G^{c}$ contains at most two edges. Suppose to the contrary that $G^{c}$ contains at least three edges, say $e_{1},e_{2},e_{3}\in E(G^{c})$. Let $e_{i}=x_{i}y_{i}$ for $i=1,2,3$. By Lemma \ref{one}, $e_{1}$ and $e_{2}$ have a common end-point. W.l.o.g., assume that $y_{1}=x_{2}$. By Lemma \ref{one}, $e_{3}$ and $e_{1}$ have a common end-point, and $e_{3}$ and $e_{2}$ have also a common end-point. So either $e_{3}=x_{1}y_{2}$ or $x_{2}$ is an end-point of $e_{3}$. However, this implies $G^{c}$ contains 3-cycle or $K_{1,3}$, which contradicts Lemma \ref{two} and Corollary \ref{four}. So $G^{c}$ contains at most two edges. Thus, if $G^{c}$ contains two edges, Lemma \ref{one} implies $G=K_{n}-E(P_{3})$; if $G^{c}$ contains one edge, then $G=K_{n}-E(P_{2})$; if $G^{c}$ contains no edge, then $G=K_{n}$.

Therefore, the theorem holds.\end{proof}

\section{Remaining $k$-hypertournaments}

\begin{theorem} A graph $G$ on $n$ vertices is the $(1,2)$-step competition graph of some $k$-hypertournament $T$ with $3\leq k\leq n-1$ if and only if $G$ is $K_{n}$, $K_{n}-E(P_{2})$, $K_{n}-E(P_{3})$, or $K_{n-1}\cup K_{1}$.\end{theorem}

\begin{proof} The ``if" part follows from the proof of Lemma \ref{three} and Theorem \ref{strong}. Now we show the ``only if" part. Let $T$ be a $k$-hypertournament and $G$ be the $(1,2)$-step competition graph of $T$. We show that $G$ is $K_{n}$, $K_{n}-E(P_{2})$, $K_{n}-E(P_{3})$, or $K_{n-1}\cup K_{1}$. Similarly to the proof of ``only if " of Theorem \ref{strong}, we get $G^{c}$ contains at most two edges unless $G=K_{n-1}\cup K_{1}$.
Thus, if $G^{c}$ contains two edges, Lemma \ref{one} implies $G=K_{n}-E(P_{3})$; if $G^{c}$ contains one edge, then $G=K_{n}-E(P_{2})$; if $G^{c}$ contains no edge, then $G=K_{n}$.

Therefore, the theorem holds.\end{proof}

\section{The $(i,j)$-step competition graph of a $k$-hypertournament}

We generalize the $(1,2)$-step competition graph to the $(i,j)$-step competition graph as follows. By the definition of the $(i,j)$-step competition graph for a $k$-hypertournament $T$, we obtain that if $i\geq 1$, $j\geq 2$, then $E(C_{1,2}(T))\subseteq E(C_{i,j}(T))$. The proof of Lemma \ref{bothside} implies the following corollary.

\begin{corollary}\label{ikstep} Let $T$ be a $k$-hypertournament with $n$ vertices satisfying $3\leq k\leq n-1$ and $i\geq 1$, $j\geq 2$ be integers. Then $xy\notin E(C_{i,j}(T))$ if and only if one of the following holds:

(a) $N^{+}(x)=\emptyset$;

(b) $N^{+}(y)=\emptyset$;

(c) $N^{+}(x)=\{y\}$;

(d) $N^{+}(y)=\{x\}$;

(e) $A_{T}^{*}\{x,y\}$ contains exactly an arc $a$, and $N_{T-a}^{+}(x)\subseteq \{y\}$, $N_{T-a}^{+}(y)\subseteq \{x\}$.\end{corollary}

\begin{theorem} Let $T$ be a $k$-hypertournament with $n$ vertices satisfying $3\leq k\leq n-1$ and $i\geq 1$, $j\geq 2$ be integers. Then $C_{i,j}(T)=C_{1,2}(T)$.\end{theorem}

\begin{proof} Clearly, $V(C_{i,j}(T))=V(C_{1,2}(T))=V(T)$. Since $E(C_{1,2}(T))\subseteq E(C_{i,j}(T))$, it suffices to show that $E(C_{i,j}(T))\subseteq E(C_{1,2}(T))$. Let $xy\in E(C_{i,j}(T))$. Suppose $xy\notin E(C_{1,2}(T))$. By Lemma \ref{bothside}, $x$ and $y$ must satisfy one of the cases $(a)-(e)$. This contradicts Corollary \ref{ikstep}. Thus $xy\in E(C_{1,2}(T))$ and $E(C_{i,j}(T))\subseteq E(C_{1,2}(T))$.\end{proof}

\end{document}